\documentclass{amsart}
\usepackage{amssymb,latexsym,amscd,hyperref}
\usepackage{pb-diagram}

\theoremstyle{plain}
\newtheorem{theorem}{Theorem}
\newtheorem{corollary}{Corollary}
\newtheorem{lemma}{Lemma}

\newtheorem{proposition}{Proposition}

\newtheorem{algorithm}{Algorithm}
\newtheorem{conjecture}{Conjecture}
\theoremstyle{definition}
\newtheorem{definition}{Definition}

\theoremstyle{remark}

\newcommand {\Q}{{\mathbb{Q}}}
\newcommand {\C}{{\mathbb{C}}}
\newcommand {\R}{{\mathbb{R}}}

\newcommand {\N}{{\mathbb{N}}}

\newcommand {\PP}{\mathbb{P}}

\newcommand {\K}{{\mathcal{K}}}

\newcommand {\OO}{{\mathcal{O}}}

\newcommand {\idc}{{\mathfrak{c}}}
\newcommand {\idf}{{\mathfrak{f}}}
\newcommand {\idp}{{\mathfrak{p}}}
\newcommand {\idm}{{\mathfrak{m}}}

\newcommand{\Norm}        {{\mathcal N}}

\newcommand{\Gal}      {\mathop{\rm {Gal}}}

\newcommand{\Cl}      {\mathop{\rm {Cl}}\nolimits}

\newcommand{\ind}        {{\mathop{\rm ind}}}
\newcommand{\res}        {{\mathop{\rm res}}}

\newcommand{\rk}        {{\mathop{\rm rk}}}

\newcommand{\eps}{\epsilon}
\begin{document}

\title[Distribution of Number Fields with Wreath Products as Galois Groups]
{The Distribution of Number Fields with Wreath Products as Galois Groups}

\author{J\"urgen Kl\"uners}
\email{klueners@math.uni-paderborn.de}
\address{Universit\"at Paderborn, Institut f\"ur Mathematik, D-33095 Paderborn, Germany.}


\begin{abstract} 
  Let $G$ be a wreath product of the form $C_2 \wr H$, where $C_2$ is
  the cyclic group of order 2. Under mild conditions for $H$ we
  determine the asymptotic behavior of the counting functions for
  number fields $K/k$ with Galois group $G$ and bounded discriminant.
  Those counting functions grow linearly with the norm of the
  discriminant and this result coincides with a conjecture of Malle.
  Up to a constant factor these groups have the same asymptotic
  behavior as the conjectured one for symmetric groups.
\end{abstract}

\maketitle

\section{Introduction}
Let $k$ be a number field and $K=k(\alpha)$ be a finite extension of
degree $n$ with minimal polynomial $f$ of $\alpha$. By abuse of
notation we define $\Gal(K/k):=\Gal(f)$. This means that we associate
a Galois group even to a non-normal extension. Therefore the Galois group of
$K/k$ is a transitive permutation group $G\leq S_n$.  

Denote by $\Norm=\Norm_{k/\Q}$ the norm function. Let
$$Z(k,G;x):=\#\left\{K/k : \Gal(K/k)=G,\ \Norm(d_{K/k})\le
  x\right\}$$ be the number of field extensions of $k$ (inside a fixed
algebraic closure $\bar\Q$) of relative degree~$n$ with Galois group
permutation isomorphic to $G$ and norm of the
discriminant $d_{K/k}$ bounded above by $x$. It is well known that the
number of extensions of $k$ with bounded norm of the discriminant is
finite, hence $Z(k,G;x)$ is finite for all $G$, $k$ and $x\in\R$. We
are interested in the asymptotic behavior of this function for
$x\rightarrow\infty$. Gunter Malle \cite{Ma4,Ma5} has given a precise
conjecture how this asymptotics should look like. Before we can state
it we need to introduce some group theoretic definitions.
\begin{definition}
  Let $1\ne G\leq S_n$ be a transitive subgroup acting on $\Omega=\{1,\ldots,n\}$.
  \begin{enumerate}
  \item For $g\in G$ we define the index $\ind(g):= n- \mbox{ the number of orbits of $g$ on }\Omega.$
  \item $\ind(G):=\min\{\ind(g): 1\ne g\in G\}.$
  \item $a(G):=\ind(G)^{-1}$.
  \item Let $C$ be a conjugacy class of $G$ and $g\in C$. Then $\ind(C):=\ind(g)$.
  \end{enumerate}
\end{definition} 
The last definition is independent of the choice of $g$ since all
elements in a conjugacy class have the same cycle shape. We define an
action of the absolute Galois group of $k$ on the
$\bar{\Q}$-characters of $G$.  The orbits under this action are called
$k$--conjugacy classes. Note that we get the ordinary conjugacy
classes when $k$ contains all $N$-th roots of unity for $N=|G|$.
\begin{definition}
  For a number field $k$ and a transitive subgroup $1\ne G\leq S_n$ we define:
  $$b(k,G):=\#\{C : C\; k\mbox{-conjugacy class of minimal index }\ind(G)\}.$$
\end{definition}
Now we can state the conjecture of Malle \cite{Ma5}, where we write $f(x) \sim g(x)$ for 
$\lim\limits_{x\rightarrow\infty} \frac{f(x)}{g(x)} =1$.

\begin{conjecture}\label{con}(Malle)
  For all number fields $k$ and all transitive permutation groups
  $1\ne G\leq S_n$ there exists a constant $c(k,G)>0$ such that
  $$Z(k,G;x) \sim c(k,G)x^{a(G)} \log(x)^{b(k,G)-1},$$
  where $a(G)$ and $b(k,G)$ are given as above.  
\end{conjecture}
We remark that at the time when the conjecture was stated it was only
known for all abelian groups and the groups $S_3\leq S_3$ and $D_4\leq S_4$.
Let us state some easy properties of the constants $a(G)$ and $b(k,G)$
which are already given in \cite{Ma4,Ma5}. It is easy to see that
$a(G)\leq 1$ and equality occurs if and only if $G$ contains a
transposition. It is an easy exercise (see Lemma
\ref{lem:transposition}) that all transpositions are conjugated
in a transitive permutation group. Therefore we obtain $b(k,G)=1$,
if $a(G)=1$. Since the symmetric group always contains a
transposition, Malle's conjecture implies that the counting function
$Z(k,n;x)$ for degree $n$ extensions with bounded discriminant as
above behaves like $c(n) x$. The latter conjecture is proven for
$n\leq 5$, see \cite{DaHe,Bh1,Bh2}, but nothing is known for $n\geq
6$.

One result of this paper is that for every even $n$ there exists a
group $G$ such that $Z(k,G;x) \sim c(k,G) x.$ This group $G$ will be a
wreath product of type $C_2 \wr H$, where $H\leq S_{n/2}$, see
Corollaries \ref{Cor1} and \ref{Cor2}. There are mild conditions for $H$,
but those are fulfilled if $H$ is nilpotent or regular for instance.

The main results will be Theorems \ref{Satz:kranz} and \ref{mainwreath}.
Let $H$ be a permutation group which fulfills the mild conditions of
Theorem \ref{Satz:kranz}. Then the counting function of $G:=C_2\wr H$ behaves
like 
$$Z(k,C_2\wr H;x) \sim c(k,G) x.$$
Furthermore, the corresponding Dirichlet series has a simple pole at 1
and has a meromorphic continuation to real part larger than $7/8$.

Note that in \cite{Kl5} we have given a counter example to 
Conjecture \ref{con}. In these counter examples it might happen that the
$\log$-factor is bigger than expected when certain subfields of
cyclotomic extensions occur as intermediate fields.  Nevertheless, the
main philosophy of this conjecture should be true.

\section{Zeta functions, Hecke $L$--series, and ray class groups}
\label{sec:hecke}

In this section we collect some properties about Hecke $L$--series.
For a number field $k$ we denote by $\PP(k)$ the set of prime ideals
of the ring of integers $\OO_k$ of $k$. We denote by
$$\zeta_k(s) :=\prod_{\idp\in\PP(k)}\left(1-\frac{1}{\Norm(\idp)^s}\right)^{-1},\;\;\Re(s)>1$$
the Dedekind zeta function of $k$ which converges absolutely and locally uniformly for
$\Re(s)>1$. This function has a simple pole at $s=1$ and we get the following estimates.

\begin{lemma}\label{residuum}
  Let $k$ be a number field of degree $m$ with absolute discriminant
  $d_k$. Then:
  \begin{enumerate}
  \item $|\zeta_k(s)|\leq \zeta_\Q(\Re(s))^m$ for all $s$ with $\Re(s)>1$.
  \item   For all $0<\eps\leq 1$:
  $$\res_{s=1} \zeta_k(s)\leq 2^{1+m} (d_k\pi^{-m/2})^\eps \eps^{1-m} \leq 2^{1+m}d_k^\eps \eps^{1-m}.$$
  \end{enumerate}
\end{lemma}
\begin{proof}
  The first assertion is Corollary 3 in  \cite[p. 326]{Nar}. The second
  one  is Corollary 3 in \cite[p. 332]{Nar}.
\end{proof}
For an ideal $\idc\subseteq \OO_k$ we consider a character $\chi$ of the ray class group
$\Cl_\idc$, i.e. a homomorphism from $\Cl_\idc$ to $\C^*$. This character is only defined for
ideals coprime to $\idc$. Let $S:=\{\idp\in \PP(k): \idp \mid \idc\}$ be the exceptional set. For
$\idp\in S$ we define $\chi(\idp)=0$. Therefore we multiplicatively extend this character to
all ideals. Now we are able to define the Hecke $L$--series:
$$L_k(\chi,s):=\prod_{\idp\in\PP(k)}\left(1-\frac{\chi(\idp)}{\Norm(\idp)^s} \right)^{-1}.$$
As the Dedekind zeta function this product converges absolutely and locally uniformly for
$\Re(s)>1$. For further properties we refer the reader to \cite[p. 343]{Nar}. 

The Hecke $L$--series have a meromorphic continuation to the left. In the 
following we need upper estimates for
 $L_k(\chi,s)$ in strips of the form $a<\Re(s)\leq 1$. The following theorem follows
directly from
\cite[equation 5.20]{IwKo}. The proof is similar to the proof of Theorem 7.4. in
\cite[p. 350]{Nar}, where we need to apply the convexity principle
 \cite[p. 265]{Lang}.
\begin{theorem}\label{bound_heckeL}
  Let $k$ be a number field of degree $m$, $\idf$ be an ideal of $\OO_k$, $\chi$ be an character of the 
  ray class group $\Cl_\idf$, and
  $D:=d_k\Norm(\idf)$. Define $\delta:=1$ if $\chi$ is the trivial character 
  and $\delta:=0$ otherwise. Then for all $\eps>0$ and all $s$ with
  $0\leq\sigma:=\Re(s)\leq 1$ we get the following estimate:
  $$|(s-1)^\delta L_k(s,\chi)| \leq c(\eps,m)(D|1+s|^m)^{(1-\sigma)/2+\eps}.$$
\end{theorem}
We can prove the following corollary.
\begin{corollary}\label{korbound}
  With the same notations as in Theorem \ref{bound_heckeL} we get for all $\eps>0$:
  $$|L_k(s,\chi)-\frac{R(\chi)}{s-1}| \leq c(\eps,m) (D|1+s|^m)^{(1-\sigma)/2+\eps},$$
  where $R(\chi)$ denotes the residue of $L_k(s,\chi)$  at $s=1$.
  We define $R(\chi)=0$, if $\chi$  is not the trivial character.
\end{corollary}
\begin{proof}
  If $\chi$ is not trivial this is Theorem \ref{bound_heckeL}.
  For the trivial character $\chi$ with exceptional set $S$ we get:
  $$L_k(s,\chi)= \zeta_k(s)\prod_{\idp\in S} \left(1-\frac{1}{\Norm(\idp)^s}\right).$$
  Using Lemma \ref{residuum} we get for our residue:
  $$|R(\chi)|\leq \tilde c(\eps,m) d_k^\eps \mbox{ for all }\eps>0.$$
  Using Theorem  \ref{bound_heckeL} and by applying the triangular inequality we find
  a new constant $c(\eps,m)$ with
  $$(s-1)L_k(s,\chi)-R(\chi)\leq  c(\eps,m)
  (D|1+s|^m)^{(1-\sigma)/2+\eps}.$$
  Since $L_k(s,\chi)-R(\chi)/(s-1)$ is analytic in $s=1$,  we get the wanted estimate
  for small $|s-1|$ using the maximum principle.
\end{proof}

For our main results we need upper bounds for the number of cyclic
extensions of a number field $k$ which are at most ramified in a given
finite set $S$ of prime ideals. We refer the reader to
\cite[p.123-126]{Lang} for properties of ray class groups which we use
in the proof of the next theorem. In the following we denote by
$\rk_\ell(\Cl_k)$ the {$\ell$--rank} of the class group of $k$. We
remark that we need the following result only for $\ell=2$.
\begin{theorem}\label{upper_Zl_bound}
  Let $k$ be an algebraic number field of degree $m$ with $r_1$ real embeddings,
  $\ell$ be a prime number, $S$ be a finite set of prime ideals of
  $\OO_k$, and
  $$S_1:=\{\idp \in S\mid \ell \notin\idp \}.$$
  Define $$s:=\begin{cases} \rk_\ell(\Cl\nolimits_k)+|S_1| +2m& \ell>2\\
                      \rk_\ell(\Cl\nolimits_k)+|S_1|+2m +r_1& \ell=2
   \end{cases}.
$$
Then there exist at most $\frac{\ell^s-1}{\ell-1}$
$C_\ell$--extensions of $k$ which are at most ramified in $S$.
\end{theorem}
\begin{proof}
  The idea of the proof is to choose $\idm$ in such a way that all $C_\ell$--extensions
  are subfields of the ray class field of $\idm$. The infinite places are only important
  when $\ell=2$. Each real infinite place may increase the $2$--rank by at most 1.
  In case $\ell=2$ we insert all real infinite places in $\idm_\infty$ and define
  $$\idm_0:=\prod_{\idp\in S} \idp^{e_\idp},$$
  where $e_\idp=1$ for $\idp\in S_1$. For $\idp\in S\setminus S_1$ we have wild ramification
  and the following estimates are valid for arbitrary $e_{\idp}>1$.
  In the following we compute upper bounds for the $\ell$--rank of 
  $(\OO_k/\idm_0)^*$. Using the chinese remainder theorem we get:
  $$(\OO_k/\idm_0)^* \cong \prod_{\idp\in S} (\OO_k/\idp^{e_\idp})^*
  \mbox{ for }\idm_0=\prod_{\idp\in S} \idp^{e_\idp}.$$ In case
  $e_\idp=1$ we get that $(\OO_k/\idp)^*$ is the multiplicative group
  of a finite field which is therefore cyclic. This explains the
  $|S_1|$-part in our formula.  In case $e_\idp>1$ we get
  $(\OO_k/\idp^{e_\idp})^* \cong (\OO_k/\idp)^* \times
  (1+\idp)/(1+\idp^{e_\idp})$. This case can only occur when $\idp$ is
  wildly ramified and therefore lies over $\ell$. In this case the
  order of the multiplicative group of the residue field is coprime to
  $\ell$. The second factor is an $\ell$--group which can be generated
  by at most $[k_\idp:\Q_\ell]+1$ elements (see e.g. \cite{HePaPo}).
  Since
  $$\sum_{\ell\in\idp} [k_\idp:\Q_\ell] = m$$ 
  we get the worst case when all prime ideals above $\ell$ are
  contained in $S$ and all corresponding completions have degree 1. In
  that case we can estimate the contribution of those prime ideals by
  $2m$.  The contribution of the unramified extensions to the
  $\ell$--rank is estimated by the $\ell$--rank of the class group.
\end{proof}
Unfortunately we do not know good estimates for the $\ell$--rank of
the class group.  The best thing we can do in general is to bound
$\ell^{\rk_\ell(\Cl_k)} \leq |\Cl_k|$.  The latter expression can be
bounded by the following (see \cite[p. 153]{Nar}).

\begin{theorem}\label{boundclass}
  For all $\eps>0$ and all $m\in\N$ there exist constants $c(m)$ and $c(m,\eps)$ such that
  for all number fields $k/\Q$ of degree  $m$ we have:
  $$|\Cl\nolimits_k| \leq c(m) d_k^{1/2}\log(d_k)^{m-1} \mbox{ and }$$
  $$|\Cl\nolimits_k| \leq c(m,\eps) d_k^{1/2+\eps}.$$  
\end{theorem}

\section{Quadratic extensions}
The asymptotics of quadratic extensions of a number field $k$ is well
studied and known. Let us define the following Dirichlet series
corresponding to $Z(k,C_2;x)$:
$$\Phi_{k,C_2}(s) := \sum_{[K:k]=2} \frac{1}{\Norm(d_{K/k})^s}=\sum_{N=1}^{\infty} \frac{a_N}{N^s}.$$
It is known that this Dirichlet series converges for $\Re(s)>1$. Here $a_N$ is the number of
quadratic extensions $K/k$ such that $\Norm(d_{K/k})=N$. This means that $a_N\geq 0$ for all
$N\in\N$. The following theorem is proved in \cite{CoDiOl2}:
\begin{theorem}[Cohen, Diaz y Diaz, Olivier]\label{phiZ2}
  Let $k$ be a number field with $i(k)$ complex embeddings. Then we get for $\Re(s)>1$:
  $$\Phi_{k,C_2}(s)= -1 +\frac{2^{-i(k)}}{\zeta_k(2s)}\sum_{\idc \mid 2\OO_k}
    \Norm(2\OO_k/\idc)^{1-2s}\sum_{\chi} L_k(s,\chi),$$
    where $\chi$ runs over the quadratic characters of the ray class group 
    $\Cl_{\idc^2}$ and $ L_k(s,\chi)$ is the Hecke $L$--series of $k$ corresponding to $\chi$.
\end{theorem}
Using a Tauberian theorem (see e.g. \cite[p. 121]{Nar2}) the following corollary is proved
in \cite{CoDiOl2}.
\begin{corollary}[Cohen, Diaz y Diaz, Olivier]\label{phiZ2res}
$$Z(k,C_2;x) \sim 2^{-i(k)}\frac{\res_{s=1}\zeta_k(s)}{\zeta_k(2)}x,$$
where $2^{-i(k)}\frac{\res_{s=1}\zeta_k(s)}{\zeta_k(2)}$ equals the  residue in $s=1$
of $\Phi_{k,C_2}$.
\end{corollary}
Our Dirichlet series has a simple pole at $s=1$ and has a meromorphic
continuation to the left. The proof of the following theorem comes
from the properties of Hecke $L$--series.  The number of characters,
i.e. the number of summands can be bounded by the size of the ray
class group which can be bounded up to a constant term depending on
$[K:k]$ by the size of the class group of $k$. The latter one we bound by
$O_{\eps,m}(d_k^{1/2+\eps})$, where $m=[k:\Q]$. Altogether we get:
\begin{theorem}\label{bound_phi}
  $\Phi_{k,C_2}(s)$ has a meromorphic continuation for 
  $\Re(s)>1/2$. In this area it has only one pole at $s=1$ with residue
  $R(k)=\frac{2^{-i(k)}\res_{s=1}\zeta_k(s)}{\zeta_k(2)}$. Furthermore, the
  function $g_k(s):=\Phi_{k,C_2}(s)-\frac{R(k)}{s-1}$ is analytic for
  $\Re(s)>1/2$ and we get for all $\eps>0$ and $\Re(s)>1/2$:
  $$|g_k(s)| \leq c(\eps,m) (d_k|1+s|^{m})^{(1-\sigma)/2+\eps}d_k^{1/2}.$$
\end{theorem}

\section{Wreath products}
\label{sec:wr}

Let $H_1\leq S_e$ and $H_2\leq S_d$ be two transitive groups and
assume $n=ed$. Then the wreath product $H_1\wr H_2 \cong H_1^d
\rtimes H_2 \leq S_n$ is a semidirect product, where $H_2\leq S_d$
permutes the $d$ copies of $H_1^d$. For a formal definition we refer
the reader to \cite[p. 46]{DiMo}. The wreath product has a nice field
theoretic interpretation in Galois theory. Assume that we have a field
tower $L/K/k$ such that $\Gal(L/K)=H_1$ and $\Gal(K/k)=H_2$. Then we
get that $\Gal(L/k)\leq H_1\wr H_2$.

We want to study the asymptotic behavior of our counting function
$Z(k,G;x)$ for wreath products $G=H_1\wr H_2$ when we assume that we
have some information for the corresponding counting functions for
$H_1$ and $H_2$.  First results in this direction already appear in
\cite{Ma4}. The $a(G)$-part of the following lemma is \cite[Lemma
5.1]{Ma4}.
\begin{lemma}
  Let $k$ be a number field and $H_1\leq S_e,H_2\leq S_d$ be
  transitive groups. Let $G:=H_1\wr H_2$. Then
  $$a(G) = a(H_1) \mbox{ and }b(k,G) = b(k,H_1).$$
\end{lemma}
\begin{proof}
  Let $g=(h_1,h_2)\in H_1\wr H_2$ where
  $h_1=(h_{1,1},\ldots,h_{1,d})\in H_1^d$ and $h_2$ is the image of
  $g$ under the projection to the complement $H_2$. If $h_2\ne 1$ then
  $g$ interchanges at least two blocks. Therefore the number of orbits
  is at most $(d-2)e+e=(d-1)e$. On the other hand, if $h_2=1,
  h_{1,2}=\cdots=h_{1,d}=1$ then $g$ has at least $(d-1)e+1$ orbits.
  Thus we may assume that $h_2=1$ and elements with minimal index have
  the property that $d-1$ of the $h_{1,i}$ equal 1. By conjugating
  with a suitable element of type $(1,\tilde h_2)\in G$ we can assume
  that $h_{1,2}=\cdots=h_{1,d}=1$. Now let $h\in H_1$ be an element of
  minimal index $e-\ell$. Then $\ind(((h,1,\ldots,1),1))=
  n-(d-1)e-\ell=e-\ell$. This shows $a(H_1)=a(G)$. It is clear that
  $h$ and $\tilde h\in H_1$ are conjugated in $h_1$ if and only if
  $((h,1,\ldots,1),1)$ and $((\tilde h,1,\ldots,1),1)$ are conjugated
  in $G=H_1\wr H_2$. $h$ and $\tilde{h}$ are in the same
  $k$--conjugacy class if a suitable power $\tilde{h}^a$ is conjugated
  to $h$. This statement remains true in the wreath product
  representation. Therefore we get the second statement.
\end{proof}

\section{Wreath products of the form $C_2 \wr H$}
\label{sec:C2wr}

In this section we prove Conjecture \ref{con} for groups $G=C_2 \wr H$,
where we need to assume weak properties of the asymptotic function for $H\leq S_d$. The proofs
are inspired by the methods described in \cite{CoDiOl2}, where the corresponding results
were shown for $G=D_4 \cong C_2 \wr C_2$. 

Let $L/k$ be an extension with Galois group $G=C_2\wr H$. Then there exists a subfield
$K\leq L$ such that $\Gal(L/K)=C_2$ and $\Gal(K/k)=H$. In a first step of our proof we will
count all "field towers" of this type, i.e. we count all extensions $L/k$ such that there
exists an intermediate field $K$ with $\Gal(L/K)=C_2$ and $\Gal(K/k)=H$. We remark that
$\Gal(L/k)\leq C_2 \wr H$ using a theorem of Krasner and Kaloujnine \cite{KraKal}. In a second
step of the proof we show that the asymptotics of proper subgroups which occur in such field
towers is strictly less.

In \cite[Proposition 8.3]{KlMa2} we already proved the following upper bound for wreath
products of this type. We remark that we weakened the assumption  by replacing the
exponent $a(H)+\delta$ by $1+\delta$. The same proof gives the new result.

\begin{proposition} \label{upper_wreath}
 Let $k$ be a number field, $H\le S_d$ be a transitive permutation group such
 that $Z(k,H;x)\le c(k,H,\delta)\,x^{1+\delta}$ for all $\delta>0$. Then
 for any $\eps>0$ there exists a constant $c(k,C_2\wr H,\epsilon)$ such that
 $$Z(k,C_2\wr H;x)\leq c(k,C_2\wr H,\epsilon)\, x^{a(C_2\wr H)+\epsilon}\,.$$
\end{proposition}
We remark that $a(C_2\wr H)= a(C_2) =1$. Furthermore we remark that the
proof counts all fields towers $L/K/k$ as above. Therefore the same upper bound
applies.

In the following let us assume that for all $\eps>0$ we have
$$Z(k,H;x) \leq c(k,H,\eps) x^{1+\eps}.$$
We remark that using the results in \cite{KlMa2} this assumption is true for
all $p$-groups. Using results proved in \cite{ElVe} this assumption is also
true for all regular $H$, i.e. when $K/k$ is normal.

For the first step we define the corresponding counting function
$$\tilde{Z}(k,C_2\wr H;x):=
\#\{L/k\mid \exists K: \Gal(L/K)=C_2, 
    \Gal(K/k)=H, \Norm(d_{L/k}) \leq x\}.$$
Using our assumption on $H$ and Proposition \ref{upper_wreath} we get for all $\eps>0$
that 
$$\tilde{Z}(k,C_2\wr H;x) \leq c(k,H,\eps) x^{1+\eps}.$$

Let us associate the corresponding Dirichlet series to $\tilde Z(k,C_2\wr H)$. Define
$$\K_H:=\{K/k \mid \Gal(K/k)= H\}$$ and
\begin{equation}\label{eq:phi}
\Phi(s) := \sum_{K\in \K_H} \frac{\Phi_{K,C_2}(s)}{\Norm(d_{K/k})^{2s}} 
= \sum_{N=1}^{\infty} \frac{a_N}{N^s},
\end{equation}
where $\Phi_{K,C_2}(s)$ is the Dirichlet series associated to $Z(K,C_2;x)$.
Since we know that $\tilde{Z}(k,C_2\wr H;x) \leq c(k,H,\eps) x^{1+\eps}$ we get
that the Dirichlet series $\Phi(s)$ converges for $\Re(s)>1$.
\begin{theorem}\label{Satz:kranz}
  Assume that there exists at least one extension of $k$ with Galois group $H$ and
  that the following estimate holds:
  $$Z(k,H;x)=O_{k,H,\eps}(x^{1+\eps}).$$
  Then the function $\Phi(s)$ defined in equation \eqref{eq:phi} has a meromorphic
  continuation to $\Re(s)>7/8$. In this area it has exactly one pole at $s=1$.
\end{theorem}
\begin{proof}
  Using Theorem \ref{bound_phi} the result is trivial if there are only finitely many extensions
  of $k$ with Galois group $H$. We remark that $d_K$ and $\Norm(d_{K/k})$ only differ by a constant depending
  on $k$ and $H$ since $d_K=d_k^{[K:k]}\Norm(d_{K/k})$. Using our assumption we get that the Dirichlet series
  \begin{equation}    \label{eq:d_K}
  \sum_{K\in \K_H} \frac{1}{\Norm(d_{K/k})^s}
  \end{equation}
  converges absolutely and locally uniformly for $\Re(s)>1$. We consider the function
  $$g(s):=\sum_{K\in \K_H}
  \frac{\Phi_{K,C_2}(s)-R(K)/(s-1)}{\Norm(d_{K/k})^{2s}},$$
  where $R(K)$ is the residue of $\Phi_{K,C_2}$ at $s=1$.
  Using Theorem \ref{bound_phi} we get that $g_K(s):=\Phi_{K,C_2}(s)-R(K)/(s-1)$
  is an analytic function for $\Re(s)>7/8$. For all
  $\eps>0$ we derive the following estimate
  $|g_K(s)|=O_{\eps,[k:\Q]}(|d_K(s+1)^{[K:\Q]}|^{1/2+1/16+\eps})$.
  Since 
  $$2\frac{7}{8}-\frac{9}{16}=\frac{19}{16}>1\mbox{ and \eqref{eq:d_K}}$$
  we get that the Dirichlet series
  $$\sum_{K\in \K_H} \frac{g_K(s)}{\Norm(d_{K/k})^{2s}}$$
  converges absolutely and locally uniformly for
  $\Re(s)>7/8$. Therefore 
  $g(s)$ is an analytic function for $\Re(s)>7/8$. 

  Using Lemma \ref{residuum} we have $R(K) = O_{\eps,[k:\Q]}(d_K^\eps)$ for
  all $\eps>0$.  Since $d_K=d_k^{[K:k]}\Norm(d_{K/k})$ we get that
  $$\frac{1}{s-1}\sum_{K\in \K_H} \frac{R(K)}{\Norm(d_{K/k})^{2s}}$$
  converges absolutely and locally uniformly for all regions which are contained in 
  $\{s\in\C \mid \Re(s)>7/8 \mbox{ and }s\ne 1\}$. The absolute convergence of all considered
  series gives the wished result for
  $$\Phi(s)=g(s)+\sum_{K\in \K_H} \frac{R(K)/(s-1)}{\Norm(d_{K/k})^{2s}}.$$
\end{proof}
As an application of a suitable Tauberian theorem we immediately get:
\begin{corollary}
  Using the same assumptions as in Theorem \ref{Satz:kranz} we get:
  $$\tilde{Z}(k,C_2 \wr H;x) \sim \res_{s=1}(\Phi(s)) x.$$
\end{corollary}
In the following we would like to show that 
$$\tilde{Z}(k,C_2 \wr H;x) \sim Z(k,C_2 \wr H;x)$$ holds, i.e. extensions which do not have
the wreath product as Galois group do not contribute to the main term. We need some group
theory.
\begin{definition}\label{block}
  Let $G\leq S_n$ be a transitive group operating on $\Omega=\{1,\ldots,n\}$.
  Then $\Delta\subseteq \Omega$ is called a block of $G$, if
  $\Delta^g \cap \Delta\in\{\Delta,\emptyset\}$ for all $g\in G$. If $G$ only
  contains blocks of size 1 or $n$ we call $G$ primitive. Otherwise $G$ is called
  imprimitive.
\end{definition}
\index{primitiv}\index{imprimitiv}\index{Block}
We remark that a field extension $L/k$ contains non-trivial subfields if and only if
$\Gal(L/k)$ is imprimitive. The blocks containing 1 are in 1-1 correspondence to the
subfields of $L/k$.

\begin{lemma}\label{lem:transposition}
  Let $G\leq S_n$ be a  transitive group containing a transposition. Then:
  \begin{enumerate}
  \item All transpositions are conjugated in $G$, i.e. $b(k,G)=1$.
  \item $G= S_e \wr H$ for $1\ne e$, $e\mid n$ and $H\leq S_{n/e}$ transitive.
  \end{enumerate}
\end{lemma}
\begin{proof}
  The first part is \cite[Lemma 2.2]{Ma5}. If $G$ is primitive the
  second statement is \cite[Theorem 3.3A]{DiMo}. Assume that
  $\tau=(i,j)$ is a transposition of $G$ and $B$ is a minimal block of
  size larger than 1 containing $i$. Then $\tau(i)=j\in B$ since all
  the other elements in $B$ are fixed by $\tau$. Therefore $G|_B$
  contains a transposition and operates primitively on $B$ ($B$ is a
  minimal block). Therefore the operation of $G|_B$ on $B$ is
  isomorphic to $S_{|B|}$.  Let $\tilde{B}$ be a conjugated block of
  $B$. By conjugating $\tau$ we can find a transposition in $\tilde
  B$. Therefore we find $n/|B|$ different copies of $S_{|B|}$.
  Therefore $G\cong S_{|B|} \wr H$, where $H$ is the image of the natural
  homomorphism $\varphi: G \rightarrow S_{n/|B|}$ which permutes the
  conjugated blocks.
\end{proof}

Now we apply this lemma to our situation of field towers. Having a subfield $K$ with
$L/K$ of degree $e=2$ means that $\Gal(f)$ contains a block system of blocks of size 2. 

\begin{lemma}\label{lem:tower}
  Let $L/K/k$ be extensions of number fields with $\Gal(K/k)=H$ and $[L:K]=2$. Let $p$ be
  a prime which is unramified in $K/k$ and  assume $p||\Norm(d_{L/K})$. Then $\Gal(L/k)=C_2\wr H$.
\end{lemma}
Note that $p$ unramified in $K/k$ and $p||\Norm(d_{L/K})$ is equivalent to $p||\Norm(d_{L/k})$.
\begin{proof}
  $\Gal(L/k)$ contains a transposition since $p||\Norm(d_{L/K})$. Let $\tau=(i,j)$ be such
  a transposition and $B$ a minimal block of $\Gal(f)$ corresponding to $K$ which contains $i$. 
  When we apply the proof of Lemma \ref{lem:transposition} to this situation we get the 
  wanted result.
\end{proof}
We remark that we can replace the prime $p$ in the above lemma by an unramified prime
ideal $\idp\subseteq \OO_k$. This does not improve the following estimates.

In the following we would like to count all field towers $L/K/k$
counted by $\tilde{Z}(k,C_2\wr H;x)$ such that $\Gal(L/k)$ is a proper subgroup of $C_2\wr H$.
Therefore we define
$$Y(k,C_2\wr H;x):=$$
$$\#\{L/K/k\mid \Gal(L/k)\ne C_2\wr H,\Gal(K/k)=H,[L:K]=2,\Norm(d_{L/k})\leq x\}.$$
We find upper bounds for this function when we count all field towers $L/K/k$  which do not satisfy
the assumptions of Lemma \ref{lem:tower}. Before we examine those field towers we need a
definition.
\begin{definition}
  Let $a\in\N$ be a positive integer and $S\subseteq \PP$ be a set of primes. Then
  $a^S$ is defined to be the largest divisor of $a$ coprime to $S$.
\end{definition}
For a field tower $k\subset K \subset L$ we get:
$$\Norm(d_{L/k})=\Norm(d_{K/k}^2)\Norm(d_{L/K}) 
\geq \Norm(d_{K/k}^2) \Norm(d_{L/K})^{S_K},$$
where $S_K:=\{p\in\PP\mid p | \Norm(d_{K/k})\}$.
We define
$$\hat Z^{S_K}(K,C_2;x):=\#\{L/K \mid \Gal(L/K)=C_2, \Norm(d_{L/K})^{S_K}\leq x, $$
$$p \mid (\Norm(d_{L/K}))^{S_k} \Rightarrow p^2 \mid (\Norm(d_{L/K}))^{S_k} \forall p\in\PP\}$$
and get
$$Y(k,C_2\wr H;x) \leq \sum_{K\in \K_H(x^{1/2})} \hat Z^{S_K}(K,C_2;x/\Norm(d_{K/k}^2)),$$
where $\K_H(x):=\{K\in \K_H \mid \Norm(d_{K/k}) \leq x\}$.
We need an estimate for $\hat Z^{S_K}(K,C_2;x)$. We denote by $a_N$ the number of fields $L$
such that $\Norm(d_{L/K})^{S_K}=N$. Since we ignore all primes in $S_K$ and all other prime
divisors occur with multiplicity at least 2, we get that $a_N=0$ if there exists a prime
$p$ with $p||N$. We choose $S\subseteq \PP(K)$ as the smallest set containing all prime ideals
which lie over a prime in $S_K$ or over a prime dividing $N$. We are interested in the number
of quadratic extensions of $K$ which are at most ramified in prime ideals contained in $S$.
We get $|S|\leq (\omega(N)+|S_K|)t$, where $\omega(N)$ is the number of different prime factors
and $t:=[K:\Q]$.
Using Theorems \ref{upper_Zl_bound} and \ref{boundclass} we get
$$a_N \leq 2^{\rk_2(\Cl_K)}2^{t (\omega(N)+|S_K|)}2^{3t} \leq c(t,\eps)d_K^{1/2+\eps}2^{t \omega(N)}.$$ 
Therefore we get:
$$\sum_{N\leq x} a_N \leq c(t,\eps)d_K^{1/2+\eps} \sum_{N\leq x^{1/2}} 2^{t\omega(N)}.$$
Using $\sum_{N\leq x}(2^t)^{\omega(N)} =O(x^{1+\eps})$ we get with a new constant $c(t,\eps)$:
$$\hat Z^{S_K}(K,C_2;x) \leq c(t,\eps)d_K^{1/2+\eps} x^{1/2+\eps}.$$
Inserting this in the above estimate for $Y(k,X_2\wr H;x)$ we get using $d_K=d_k^2\Norm(d_{K/k})$:
$$Y(k,C_2\wr H;x)
\leq \sum_{K\in \K_H(x^{1/2})} c(t,\eps)(d_k^2\Norm(d_K))^{1/2+\eps} 
\left(\frac{x}{\Norm(d_{K/k}^2)}\right)^{1/2+\eps} $$
$$\leq c(t,\eps) d_k^{1+2\eps}x^{1/2+\eps}  
\sum_{K\in \K_H(x^{1/2})} \frac{\Norm(d_{K/k})^{1/2+\eps}}{\Norm(d_{K/k})^{1+2\eps}}$$
Using $\Norm(d_{K/k})\leq x^{1/2}$ we get:
$$Y(k,C_2\wr H;x)
\leq c(t,\eps) d_k^{1+2\eps}x^{1/2+\eps} x^{1/4+\eps} \sum_{K\in \K_H(x^{1/2})}
\frac{1}{\Norm(d_{K/k})^{1+2\eps}}.$$
The last sum converges under the assumption for $H$ of Theorem \ref{Satz:kranz}.
This proves for all $\eps>0$ the following estimate:
$$Y(k,C_2\wr H;x) \leq c(k,H, t,\eps) x^{3/4+2\eps}.$$
Since $Z(k,C_2\wr H;x) + Y(k,C_2\wr H;x) =\tilde{Z}(k,C_2\wr H;x)$ and Theorem 
\ref{Satz:kranz} we proved the following:
\begin{theorem}\label{mainwreath}
  Assume the same as in Theorem \ref{Satz:kranz}. Then the Dirichlet
  series corresponding to $Z(k,C_2\wr H)$ has a meromorphic
  continuation to $\Re(s)>7/8$, where $s=1$ is the only pole in that
  region. The residue $r$ of that pole coincides with the one of the
  function $\Phi(s)$. We get:
  $$Z(k,C_2\wr H;x) \sim \res_{s=1}(\Phi(s))x.$$
\end{theorem}
We are able to give an expression for this residue as a convergent sum.
\begin{corollary}
  $$\res_{s=1}(\Phi(s)) = \sum_{K\in \K_H} \frac{\res_{s=1}\zeta_K(s)}{2^{i(K)}d_K^2\zeta_K(2)}.$$
\end{corollary}
These results support our main conjecture.
\begin{corollary}\label{Cor1}
  Conjecture \ref{con} is true for all 
  $C_2\wr H$ and all number fields  $k$ such that $H$ fulfills the assumptions of Theorem
  \ref{Satz:kranz}.
\end{corollary}
We are already remarked that this assumption is true for all $p$--groups and all 
permutation groups in regular representation. Therefore we get the following
corollary.

\begin{corollary}\label{Cor2}
  For even $n$ there always exists a group $G\leq S_n$ with $a(G)=1$ and 
  $$Z(k,G;x) \sim c(k,G) x =c(k,G)x^{a(G)}.$$
\end{corollary}

\section*{Acknowledgments}
I would like to thank Gunter Malle for many discussions about this topic. This project was partially supported
by the Deutsche Forschungsgemeinschaft (DFG).
\bibliographystyle{abbrv} 
\bibliography{myref} 
\end{document}